\documentclass{amsart}

\usepackage{amsthm}
\usepackage{hyperref}

\usepackage{graphicx,subfigure}

\def\qed{\hbox{${\vcenter{\vbox{		 
   \hrule height 0.4pt\hbox{\vrule width 0.4pt height 6pt
   \kern5pt\vrule width 0.4pt}\hrule height 0.4pt}}}$}}

\newtheorem{theorem}{Theorem}

\newtheorem{definition}[theorem]{Definition}

\newtheorem{lemma}[theorem]{Lemma}

\newtheorem{remark}{Remark}

\def\cF{\mathcal F}

\def\bE{\mathbb E}

\def\bP{\mathbb P}

\def\bR{\mathbb R}

\begin{document}


\title[Epidemics on networks with stochastic infection rates]{Epidemics on networks with heterogeneous population and stochastic infection rates}



\author{Stefano Bonaccorsi}
\address[Stefano Bonaccorsi]
{Department of Mathematics, University of Trento, via Sommarive 14, 38123 Povo (Trento), Italia}
\email{stefano.bonaccorsi@unitn.it}

\author{Stefania Ottaviano} 
\address[Stefania Ottaviano]{CREATE-NET, Via alla Cascata 56/d, 38123 Povo (Trento), Italy and Department of Mathematics, University of Trento, via Sommarive 14, 38123 Povo (Trento), Italy}
\email{sottaviano@create-net.org}

%
%
%

		%
		
\begin{abstract}
In this paper we study
the diffusion of an SIS-type epidemics on a network under the pre\-sence of a random environment, that enters in the definition of the infection rates of the nodes.
Accordingly, we model the infection rates in the form of independent stochastic processes.
\\
To analyze the problem,
we apply a mean field approximation, which allows to get a stochastic diffe\-rential equations for the
probability of infection in each node, and classical tools about stability, which require to find suitable Lyapunov's functions.
\\
Here, we  find  conditions which guarantee, respectively, extinction and  stochastic persistence of the epidemics.
We show that there exists two regions, given in terms of the coefficients of the model, {one where the system
goes to extinction almost surely, and the other where it is stochastic permanent}. These two regions are, unfortunately, not adjacent,
as there is a gap between them, {whose extension depends on the specific level of noise}. In this last region, we perform numerical analysis to suggest the true behavior of the solution.

\end{abstract}



\maketitle

\section{Introduction}

Epidemiology studies the spread of diseases in (human) population, where the brackets mean that
similar problems arise in real world problems such as, e.g., virus propagation in computer networks { and diffusion of information}.

{The
epidemic spreading is governed by an inherently probabilistic
process. 
The simplest stochastic model formulations
are based on discrete and continuous time Markov chains \cite{allen2008introduction,naasell2002,Britton2010}. In turn, using Markov processes, we can obtain a deterministic approximation for the whole trajectory in the case of large population sizes \cite{kurtz1971,Britton2010} and, over the years, works on deterministic models have dominated strongly over works on stochastic models, because of their greater simplicity and tractability \cite{naasell2002}. However even if they provide an appropriate way to represent certain situations
of interest, a correct analysis should consider explicitly the stochastic
nature of epidemic spreading, especially when dealing
with small populations  \cite{Britton2010, PietSurvey}.} 
{Also stochastic differential equations have been used to approximate the Markov chain model (see for instance \cite{pollett2001diffusion,tornatore2005stability,ding2008asymptotic, McCormack2006, Dargatz2006, allen2008introduction}).}\\
{ An additional simpler approach, for modeling random fluctuations}, consists of introducing parameter perturbations in the ordinary differential equations (see e.g. \cite{Mao2002, Gray2011,Spagnolo2004, Stanescu2009} ). Indeed,
 the parameters have a great variability depending on errors in the observed and measured data, on uncertainties, e.g. when some variables cannot be measured, on lack of knowledge or, simply, on the presence of a random environment. 
Hence, it is appropriate to consider the parameters as random variables, with a specified  distribution, and study differential equations with random coefficients or,  incorporating stochastic effects {instead, 
considering stochastic equations whose parameters depend on the random environment}.
{In this regard it is useful to reveal how the noise affects the population system \cite{li2009population}.
In the literature, we can find several examples 
of processes that can be used as environmental noise, like the Brownian motion or
the telegraph noise, where the parameters switch from one set to another according
to a Markov switching process (see e.g.  \cite{Luo2007,takeuchi2006evolution,li2009population})}.\\\\

 {\subsection{The population structure.}
It is worth mentioning that in most of the classical epidemic model,
the basic approach is to consider an \textit{ homogenous mixing approximation}, meaning
that individuals in the population are well mixed and interact with each other
completely at random.}
{This assumption, however, is very strong and disputable, both in deterministic
and stochastic models, 
since details such as geographical location, presence of community
structures, or the specific role of each individuals in the epidemic spreading
are ignored.\\
An implicit assumption, in the homogeneous approximation, 
is that each infected
individual has a small chance to infect {\em every} susceptible individual in the population.
Conversely, diseases spread through a network of social contacts, thus
the epidemic has a much higher probability of spreading to a limited set
of susceptible contacts \cite{keeling2005}, and the dynamics of disease transmission
strongly depends on the properties of the population contact network.}
{Hence, social contact patterns among individuals 
is a key ingredient in
the realistic characterization and modeling of epidemics \cite{fumanelli2012inferring}.}\\ 
{We need to consider also that
the effective diffusion occurs through individual people, that may differ in many different aspects (genetics,
biology or social behavior), thus the parameters characterizing the infection (say, the rate of infection or the rate of recovery) have a variety among the population and,
 in most cases they cannot be fixed a priori: only their statistical properties are known. A short overview on works in literature that consider heterogeneous populations can be found in \cite{widder2014heterogeneous, qu2015sis}}.
 \\

{Recent papers \cite{qu2015sis,buono} have taken into account epidemic spreading on networks with time-constant, heterogeneous parameters, considering these parameters as random variables with given distribution.} \\
{However since most realistic heterogeneous parameter distributions are not constant in time, in our work we consider them as additional dynamical variables.}\\
{Definitively we consider an
{\em heterogeneous population}, {taking into account the fluctuation in time of the transmission of infection, for an epidemics that spread on a given population contact network}.}\\


{\subsection{Outline and main results}}
 
Through the paper we shall discuss the dynamics of the following system of 
It\^o stochastic differential equations
\begin{equation}
\label{SDE1}
\begin{aligned}
{\mathrm d} x_i(t) =& \left[\beta s_i(t) (1 - x_i(t)) - \delta x_i(t)\right] \, {\mathrm d}t
 + \sigma_i(x_i(t)) s_i(t) (1 - x_i(t)) \, {\mathrm d}w_i(t), 
\\
s_i(t) =& \sum_{j=1}^N a_{ij} x_j(t), \qquad i \in \{1, \dots, N\}
\end{aligned}
\end{equation}
for the unknowns $x_i(t)$ which represent the probability that individual $i$ is infected at time $t$,
with a given vector of initial conditions $(x_1(0), \dots, x_N(0))$.\\

 The paper is organized as follows.
In Section \ref{dyn} we consider the existence of the solution for all times. Typically, 
conditions assuring the non-explosion of the solution involves local Lipschitz continuity and
a linear growth condition. In our case, we miss this last condition so it is 
necessary to prove that the solution does not explode at a finite time. {We prove also that the unique global solution remains within $(0,1)^N$ whenever it starts from this region}.
\\

Then, In Section \ref{s3}
we concentrate on the long time behavior of the solution. 
We prove that there are two typical phenomena arising in accordance with the relative values of the parameters
of the model: the mean value of the infection rate $\beta$, the recovery rate $\delta$ and the intensity of the noise. 
We show that in the case that $\delta$ is large with respect to the other parameters, then
the solution tends to extinction almost surely. We obtain this result by proving the global attractiveness of the null solution for system \eqref{sdevect} in Section \ref{s3}.
We use a standard technique of Lyapunov functions, introduced by Khasminskii \cite{Khasminskii1980} and later 
employed by many other authors (see for instance \cite{Mao1997, Gard1988}).
Later, in Section \ref{s2.1} we discuss on the {\em stochastic permanence} of the solution.
This concept,  which can be paraphrased by saying that the epidemic process will survive forever, 
is one of the most important and interesting topics in the analysis of the model. 
In section \ref{s2.1} we give conditions on the parameters, that informally can be stated by saying that $\delta$ (the recovery rate) is small compared to the other parameters, such that
the epidemic process is stochastically permanent.
\\
 Between the two regions of extinction and permanence we have a gap, given in terms of the parameters of the model, thus, in Section \ref{num} we provide numerical results that also investigate the behavior of the solution into this middle region.

\section{Description of the model}

In this paper, we consider
a continuous-time SIS (susceptible - infected - susceptible) model for the diffusion
of epidemics spreading on an undirected simple graph $G=(V,E)$, {thus if a node $i$ is in contact with node $j$ and consequently $i$ can influence $j$, also the node $j$ can influence node $i$. At this stage we consider that the infection rate is equal for all pair of nodes, but later we assume the possibility to have different infection rates.} 
The population is described via an $N \times N$ contact matrix $A$ (i.e. the adjacency matrix of $G$);
for simplicity, we do not allow for birth, death and movement of the individuals 
{(we underline, however, that our model can be easily adapted to cover the case of a time-varying contact matrix $A = A(t)$).
Although we refers to epidemic diffusion in human population, we shall emphasize that the model
 is well suited to represent also man-made architectures}.

In the SIS model, an individual can be repeatedly infected, recover and yet be infected again. This model covers those types of disease that does not confer immunity, e.g. common cold, sexually transmitted
diseases, and other bacterial infections \cite{garnett1996sexually, lajmanovich1976deterministic}. 
Computer viruses also belong in this category, indeed once cured, without a constant upgrade of the anti-virus software, the computer has no way to fend off subsequent attacks be the same virus. 
\cite{newman2003structure}.
The SIS model can be used also for describing some social behaviors and emotions \cite{hill2010emotions}.

 The recovery process is a Poisson process with rate $\delta$, and the infection process is a \textsl{per link} Poisson process where the infection rate between 
an healthy and an infected node is $\beta$.
Thus each susceptible node can be infected, at time $t$, by any of its neighbors, with a {total infection rate} $\beta s_i(t)$, where $s_i(t)$ stands for
the strength of the epidemics in the nodes that are directly connected with it. 
{All the infection and recovery processes are independent, thus they compete for the production of an event (infection or recovery)}.

{The state of the collective system of all nodes, i.e. the state of the network,
is actually the joint state of all the nodes' states.
Since we assume that the infection and curing processes are of Poisson
type, the SIS process, developing on a graph with N nodes, can be modeled
as a continuous-time Markov process with $2^N$ states, covering all possible
combinations in which N nodes can be infected \cite{VanMieghem2009,simon2011}.

\subsection{Mean field approximation} \label{mfa}

Unfortunately, the exponential growth with $N$ of the state space poses severe limitations in order to determine the set of solutions for large, real networks.} For this reason,
 in \cite{VanMieghem2009} a first order mean-field approximation (NIMFA) of the exact model is proposed. \\
Basically NIMFA replaces
the actual infection rate for the node $i$,
$\beta \sum\limits_{i=1}^N a_{ij} X_j(t)$
(where the sum is done on all the neighbor nodes),
by its average rate
\begin{align*}
\beta \sum\limits_{i=1}^N a_{ij} \bE[X_j(t)], 
\end{align*}
where $X_j(t)$ is the state variable representing the state of the node $j$.

Thus, we introduce the unknown $x_i(t)$ which represents the probability that individual $i$ is infected at time $t$, $x_i(t) = \bE[X_i(t)]$ and 
in the NIMFA model their evolution obeys the following differential equation
\begin{align}\label{nimfa}
\dot x_i(t) =& \beta s_i(t) (1 - x_i(t)) - \delta x_i(t), \qquad i \in \{1, \dots, N\}
\\
s_i(t) =& \sum_{j=1}^N a_{ij} x_j(t)
\end{align}
where $s_i(t)$ represents the strength of the infection that may reach individual $i$ through the adjacency matrix $A = \big( a_{ij} \big)$.
\\
The time-derivative of the infection probability of individual $i$ consists of two competing processes:
\begin{enumerate} 
\item while healthy (with probability $1 - x_i(t)$), all infected neighbors, whose average number is $s_i(t)$, try to infect node $i$ at rate $\beta$;
\item while individual $i$ is infected (with probability $x_i(t)$) it is cured at rate $\delta$.
\end{enumerate}
\textit{Epidemic threshold.} 
{The exact SIS Markov model has a unique absorbing state, so in the long run we shall get the extinction of the epidemics.
However the waiting time to absorption is a random variable whose distribution depends on the initial state of the system, and on the parameters of the model \cite{allen2008introduction,naasell1996quasi, naasell2002,Bonaccorsi2015}. 
In fact there is a critical value $\tau_c$ of the effective infection rate $\tau= \beta/\delta$, whereby if $\tau$ is distinctly larger than $\tau_c$ the time to absorption  grows exponentially in $N$,
 while for $\tau$ distinctly less than $\tau_c$ the lifetime of epidemic is rather small \cite{naasell1996quasi, ganesh2005,VanMieghem2009}}. 
 
{Thus, above the threshold a typical realization of the epidemic process may experience a very long  time before absorption to the zero-state. 
During this period, the distribution of the number
of infected individuals is close to the distribution of the
number
of infected individuals  conditioned on non-extinction, the so-called 
 {\em quasi-stationary distribution} \cite{seneta1967quasi, naasell1996quasi, naasell2002}. 
However, the analytical computation of the exact quasi-stationary distribution is not feasible, as it is showed in \cite{naasell1996quasi}}. 
\par\noindent
{Even so, numerical simulations of SIS processes reveal that, already for reasonably small networks $(N \geq 100)$ and when $\tau > \tau_c$, the overall-healthy state is only reached after an unrealistically long time. 
Hence, the indication of the model is that, in the case of real networks, one should expect that the extinction of epidemics is hardly ever attained \cite{VanMieghem2013, Draief2010}.
For this reason the literature is mainly concerned with establishing the value of the epidemic threshold, being a key parameter behind immunization strategies related to the network protection against viral infection.}

 {NIMFA determines the epidemic threshold for the effective infection rate as $\tau^{(1)}_c =\frac{1}{\lambda_{1}(A)}$, where the superscript $(1)$ refers to the first-order mean-field approximation \cite{VanMieghem2009,VanMieghem2014,Bonaccorsi2015}.}\\
{When $\tau > \tau^{(1)}_c$, the mean-field equation \eqref{nimfa} shows 
a second non-zero steady-state that reflects well the observed viral behavior \cite{VanMieghem2012a}: it can be seen as the analogous of the quasi-stationary distribution 
of the exact stochastic SIS model.}\\

A consequence of replacing the random variable
by its expectation in the mean field approximations results in the neglecting of correlations
in the dynamic process: mean field approximation treats the joint probability
$\bP(X_i = 1,X_j = 1)$ as if $X_i$ and $X_j$  were independent,
$\bP(X_i = 1) \bP(X_j = 1)$. 
As a consequence of this simplification, NIMFA yields an upper bound for the probability of infection of each node, as well as a lower bound for the epidemic threshold, i.e., $\tau_c \ge \tau_c^{(1)}$. 
This fact was rigorously proved in \cite{CatorPositivecorrelations}, by showing that the state of nodes are non-negatively correlated.  
Thus from the applicative standpoint, a key issue is to determine for which networks of given order NIMFA performs worst, meaning that $\alpha =\frac{\tau_c}{\tau_c^{(1)}}$ is largest.
One can observe that,
basically, if the 
states of the nearest nodes are sufficiently weakly dependent and the number of
neighbors of node $i$, $d_i$ (i.e., the degree of node $i$) is large enough so that the {Lindberg's} Central
Limit Theorem 
is applicable, then there is a good approximation result by previous replacement. 
Informally, we can say that the mean field approximation
holds if in the underlying network, the degree of the nodes increase as
the number of nodes $N$ tends to infinity \cite{accuracy}. 
\\
Moreover, evaluations on the variance of $\beta \sum\limits_{i=1}^N a_{ij} X_j(t)$ in \cite{VanMieghem2009}, shows that  the deviations between
the NIMFA model and the exact SIS are largest for intermediate
values of $\tau$.

{As we said in the previous section, NIMFA assumes that the infection rate (such as the recovery rate) is given a \textit{priori}, it is constant among the population and does not change in time}.
 We now modify the reference model to allow heterogeneity of the population, {and considering that the reference parameters can be affected by stochastic fluctuations.} 

{Let $\Omega$ be a given reference sample space}, and $\omega \in \Omega$ one possible outcome that represents a possible perturbation of the population's parameter.
We consider that the disease spreads among the population according to the dynamics
\begin{equation}
\label{e:080615-1}
\begin{aligned}
\dot x_i(t,\omega) =& \beta(\omega) s_i(t,\omega) (1 - x_i(t,\omega)) - \delta(\omega) x_i(t,\omega), \qquad i \in \{1, \dots, N\}
\\
s_i(t,\omega) =& \sum_{j=1}^N a_{ij} x_j(t,\omega)
\end{aligned}
\end{equation}
Hence, the rate coefficients, $\beta(\omega)$ and $\delta(\omega)$, as well as the unknowns $x_i(t,\omega)$s, are assumed to be random variables on the probability space $(\Omega, \cF, \bP)$, for a given $\sigma$-algebra
$\cF$ and a probability measure $\bP$ on it.

{However  realistic heterogeneous parameter distributions are not constant in
time but could be considered as additional dynamical variables \cite{widder2014heterogeneous}.
Thus, in order to model the fluctuations in time of the parameters, we consider white noise \cite{Arnold1974}, a natural starting point for the case when the functional
form and properties of the stochastic process are not known (\cite{widder2014heterogeneous,Mao2002,Gray2011}).}

In particular, we
consider that each node {$i$} can be infected, {at time $t$}, by all its infected neighbors with rate $\beta_i(t)$, that is described by a stochastic process of the form
\begin{align*}
\beta_i(t) \longrightarrow \beta + \sigma_i(x(t)) \, \dot w_i(t),
\end{align*}
where $\dot w_i(t)$ is the white-noise mapping and the functions {$\sigma_i: \mathbb R \rightarrow [0, + \infty)$}, {that provide the noise
level for each node}, 
are 
{locally Lipschitz continuous and satisfy}
 \begin{equation}\label{cond}
\sup_{x \in (0,1)} \frac{\sigma_i(x)}{x} \le M, \quad \text{for every\ } i = 1, \dots, N.
\end{equation}
In order to make things formal, we shall introduce a standard $N$-dimensional Brownian motion
$W(t) = \left(w_1(t), \dots, w_N(t) \right)$  
defined on a stochastic basis $(\Omega, \cF, \{\cF_t\}, \bP)$ and interpret $\dot w_i(t)$ as ${\rm d}w_i(t)$ in the It\^o sense.

It is important to discuss the model for the diffusion coefficient.
Our choice implies that the intensity of the infection rate varies around a mean value, and the disturbance is small if the 
value of the probability of infection is small. {Notice that, choosing a diffusion coefficient independent of the epidemic level would have implied
a random variability of the solution also near the zero level, thus allowing the solution to go below zero, 
region that has no physical meaning.}

{The dependence of the diffusion coefficient on the solution is typical in papers concerning 
population dynamics with environmental noise. 
E.g., in \cite{Luo2007}, 
the noise intensities depend linearly
on the population sizes, and clearly this assumptions satisfy the condition \eqref{cond}}.
{When random environment is reflected in fluctuations in the transmission
and recovery of the infection or, e.g, in parameters describing the interactions between many species, like in the Lotka-Volterra equations, the multiplicative noise term is considered appropriate (see e.g. \cite{Mao2002,mao2003asymptotic, Luo2007, Gray2011}). An example of stochastic system with additive noise, reflecting the presence of an eventual outside source of infection, can be found in \cite{widder2014heterogeneous}.}

On the other hand, for the sake of simplicity we shall assume that the recovery rate $\delta$ is a deterministic constant. 
The general case does not change substantially the results that we present here.
\\

The system is then described by the It\^o stochastic differential equation
\begin{equation}
\label{e:090615-1}
\begin{aligned}
{\mathrm d} x_i(t) =& \left[\beta s_i(t) (1 - x_i(t)) - \delta x_i(t)\right] \, {\mathrm d}t
 + \sigma_i(x_i(t)) s_i(t) (1 - x_i(t)) \, {\mathrm d}w_i(t), 
\\
s_i(t) =& \sum_{j=1}^N a_{ij} x_j(t), \qquad i \in \{1, \dots, N\}
\end{aligned}
\end{equation}
with a given vector of initial conditions $(x_1(0), \dots, x_N(0))$.
\\
We represent previous equation also by the vector-valued stochastic differential equation
\begin{equation}\label{sdevect}
\begin{aligned}
&{\mathrm d} X(t)=  f(X(t)) \, {\mathrm d}t + g(X(t)) \, {\mathrm d}W_t 
\\
& X(0) = (x_1(0), \dots, x_N(0)),
\end{aligned}
\end{equation}
where $X(t) = (x_1(t), \dots, x_N(t))$ while
$f(X(t))$ and $g(X(t))$ are functions defined in $\mathbb{R}^N$ and $L(\bR^N,\bR^N)$, respectively. The $j$-th component of $f$ is $\beta (1- x_j(t)) s_j(t) -\delta x_j(t)$, whereas $g$ is a diagonal matrix with  entries   $\sigma_j(x_j(t)) (1- x_j(t)) s_j(t)$.
We shall denote $\Delta = (0,1)^N$. 

\section{Basic dynamics of the stochastic model}\label{dyn}

\begin{theorem}\label{th:1}
For any initial condition $X(0) = (x_1(0), \ldots, x_N(0))$ such that $X(0) \in \Delta$, there exists a unique global solution to system \eqref{e:090615-1}
on $t \geq 0$ and the solution remains in $\Delta$ almost surely for all times.
\end{theorem}

One way to ensure  that the dynamics remain on $\Delta$ (i.e., $\Delta$ is invariant)  is  to note that
for any $x \in \partial \Delta$ the scalar product between $f(x)$ and the outward normal vector $\nu(x)$ is
\begin{align*}
\langle f(x), \nu(x) \rangle \le 0, \qquad x \in \partial \Delta,
\end{align*}
and that the diffusion terms is orthogonal to the outward normal vector $\nu(x)$:
\begin{align*}
\langle g(x) \cdot y, \nu(x) \rangle = 0, \qquad \text{for all\ } y \in \bR^N,\ x \in \partial \Delta.
\end{align*}
However, the proof below follows a different approach, which allows to prove some further properties of the solution.

\begin{proof}
Since the coefficients of the equation are locally Lipschitz continuous, for any given initial value $X(0) \in (0,1)^N$
there is a unique local solution 
on $t \in [0, \tau_e)$, where $\tau_e$ is the explosion time (see for instance \cite{Arnold1974}). 
\\
To show this solution is global, we need to show that $\tau_e = \infty$ a.s.
This is achieved if we prove a somehow stronger property of the solution, namely that it never 
{leaves the domain $\Delta$}.
The following computations are somehow standard (compare for instance \cite{Luo2007}).
Let $n_0 > 0$ be sufficiently large for  $x_i(0)\in \left(\frac{1}{n_0}, 1-\frac{1}{n_0} \right)$ for all $i=1, \dots N$.   
For each integer $n \geq n_0$, define the stopping time
\begin{align*}
\tau_n = \inf \left\{t \in [0, \tau_e) : \min_{1 \le i \le N} x_i(t) \leq 1/n   \text{  or  } \max_{1 \le i \le N} x_i(t)\geq 1- 1/n\right\},
\end{align*}%
where, as customary, $\inf \emptyset =+\infty$ (with $\emptyset$ denoting the empty set). 
\\
Clearly $\tau_n$ is increasing as $n \rightarrow \infty$ and letting $\tau_{\infty}= \lim_{n \rightarrow \infty} \tau_n$, 
we have $\tau_\infty \leq \tau_e $ a.s. 
Hence we basically need to show that $\tau_\infty=\infty$ a.s;  if this were not so, there would exists a pair of constants $T > 0$ and $\epsilon \in (0,1)$ such that
\begin{align*}
\bP \left\{\tau_{\infty} \leq T \right\} > \epsilon.
\end{align*}
Accordingly, there is an integer $n_1 \geq n_0$ such that
\begin{equation}
\label{Peps}
\bP \left\{ \tau_n \leq T \right\} \geq \epsilon \qquad \forall n \geq n_1.
\end{equation}

Now we define a function $V: (0,1)^N \rightarrow \mathbb{R}^+$ as 
\begin{align*}
V(X(t))= - \sum_{i=1}^N \log\left[x_i(t) (1-x_i(t))\right].
\end{align*}
By It\^o's formula we have
\begin{equation}\label{ito}
\begin{aligned}
{\rm d} V(X(t))= &\sum_{i=1}^N \left(\frac{1}{1-x_i} -\frac{1}{x_i}\right)\left[\left( \beta s_i (1-x_i)   - \delta x_i \right) \, {\rm d} t+
\sigma_i(x_i) s_i (1 - x_i) \, {\rm d}w_i(t) \right]
 \\
&\quad
+ \frac{1}{2} \sum_{i=1}^N \left( \frac{1}{(1-x_i)^2} + \frac{1}{x_i^2}\right)\sigma^2(x_i) s_i^2(1-x_i)^2 \, {\rm d}t,
\end{aligned}
\end{equation}
where we hide the explicit dependence on time of the processes $x_i$ and $s_i$.
Let $L$ be the infinitesimal generator associated to the stochastic equation \eqref{sdevect} defined, for $V \in C^\infty(\Delta)$, by
\begin{equation}
\label{gen}
L V(X) = \sum_{i=1}^N f_i(X) \partial_{x_i} V(X) + \frac{1}{2} \sum_{i=1}^N g_{ii}^2(X) \partial^2_{x_i x_i} V(X), \qquad X = (x_1, \dots, x_N); 
\end{equation}
then from \eqref{ito} 
\begin{equation*}
{\rm d} V(X(t)) = L V(X(t)) {\rm d}t + {\rm d}M(t),
\end{equation*}
where $M(t)$ is the (local) martingale defined by
\begin{align*}
M(t) = \sum_{i=1}^N \int_0^t 
\left(\frac{1}{1-x_i(t)} -\frac{1}{x_i(t)}\right)\sigma_i(x_i(t)) s_i(t) (1 - x_i(t)) \, {\mathrm d}w_i(t).
\end{align*}

\begin{lemma}
\label{l:201115.1}
There is a finite constant $K$ such that
$L V(X) \le K$ for every $X \in \Delta$.
\end{lemma}

We postpone the proof of the lemma and continue to pursue the global existence of the solution.
By the lemma we have
\begin{equation}
\int_{0}^{ \tau_n \wedge T} {\rm d}V(X(t)) 
\leq 
\int_{0}^{ \tau_n \wedge T} K \, {\rm d}t + M(t),
\end{equation}
and taking the expectation 
\begin{equation}\label{diseg}
 \bE [V\left(X(\tau_n \wedge T)\right)] \leq \bE[V(X(0))] + K \, \bE(\tau_n \wedge T)  
 \leq \bE[V(X(0))]  + KT.
\end{equation}

Set $\Omega_n=\left\{\tau_n \leq T\right\}$ for $n \geq n_1$. 
By \eqref{Peps} we have $P(\Omega_n) \geq \epsilon$. 
Since for every $\omega \in \Omega_n$, there is at least one of the $x_i(\tau_n, \omega)$ equaling either $1/n$ or $1-1/n$, then it holds
\begin{equation}\label{diseg2}
V\left( X(\tau_n, \omega)\right) \geq -\left(\log\left(\frac{1}{n}\right) + \log\left(1-\frac1n\right)\right).
\end{equation}
Then from \eqref{diseg} and \eqref{diseg2} it follows that
\begin{equation*}
V(X(0))  + KT \geq \bE \left[\chi_{\Omega_n} V\left(X(\tau_n, \omega)\right)\right] \geq 
\epsilon \left(\log\left(n\right) + 1 \right)
\end{equation*}
where $\chi_{\Omega_n}$ is the indicator function of $\Omega_n$. Letting $n \rightarrow \infty$ we have the following contradiction
\begin{align*}
\infty > V(X(0))  + KT = \infty,
\end{align*}
hence we must have $\tau_\infty=\infty$ a.s. and the proof is complete.
\end{proof}

\begin{proof}[Proof of Lemma \ref{l:201115.1}]
Recall that
\begin{equation*}
L V(X)=  \sum_{i=1}^N \left[\left(\frac{1}{1-x_i} -\frac{1}{x_i}\right)\left( \beta s_i (1-x_i)   - \delta x_i \right) + \frac{1}{2} \left( \frac{1}{(1-x_i)^2} + \frac{1}{x_i^2}\right)\sigma_i^2(x_i) s_i^2 (1- x_i)^2\right];
\end{equation*}
the last term is bounded by
\begin{align*}
\frac{1}{2} \left( \frac{1}{(1-x_i)^2} + \frac{1}{x_i^2}\right)\sigma^2(x_i) s_i^2 (1- x_i)^2
\le \frac{1}{2} \left( \frac{x_i^2 + (1-x_i)^2}{(1-x_i)^2 \, x_i^2}\right)M^2 \, x_i^2 \, s_i^2 (1- x_i)^2
\le M^2 (N-1)^2.
\end{align*}
The first term is given by
\begin{align*}
\left(\frac{1}{1-x_i} -\frac{1}{x_i}\right)\left( \beta s_i (1-x_i)   - \delta x_i \right) 
= \frac{2x_i - 1}{x_i (1-x_i)} \left( \beta s_i (1-x_i)   - \delta x_i \right);
\end{align*}
since the function 
\begin{align*}
y(x) = \frac{2x - 1}{x (1-x)} \left( \beta s (1-x)   - \delta x \right), \qquad x \in (0,1)
\end{align*}
has a maximum in
$x_m = \frac{\sqrt{\beta s}}{\sqrt{\beta s} + \sqrt{\delta}}$
that is
\begin{align*}
y(x_m) = \beta s + \delta - 2 \sqrt{\delta \beta s},
\end{align*}
in our framework, since $s_i = \sum a_{ij} x_j \le N-1$, we finally get
\begin{align}\label{e:201115.6}
L V(X) \le N \left[ \beta (N - 1) + \delta + M^2 (N-1)^2 \right]
\end{align}
so the claim follows with a constant $K$ given by the right-hand side of \eqref{e:201115.6}.
\end{proof}

\section{Long time properties of the zero solution}
\label{s3}

Now we provide an analysis of the stability of the zero solution, i.e. the disease-free equilibrium, in order to identify the threshold condition for controlling the infection or eventually eradicating it.
\\
Let $X_0=0$ be the vector of all zero
components and let us consider the equation \eqref{sdevect}. 
Since $f(X_0)=0$ and $g(X_0)=0$ for all $t \geq 0$, it follows that the unique solution of \eqref{sdevect} satisfying the initial condition $X(0)=X_0$ is the identically zero solution $X(t)=X_0$. 
For the definitions and conditions on the stability of the zero solution we refer to \cite[Chapter 11]{Arnold1974}.

%
\begin{remark}
The contact matrix $A$, that is the adjacency matrix of an undirected graph, is symmetric and
satisfies
\begin{equation}
\label{e:221115.1}
\begin{aligned}
\langle A X,& X \rangle \le \lambda_1(A) |X|^2,
\\
\langle A X,& A X \rangle \le \lambda_1(A)^2 |X|^2
\end{aligned}
\end{equation}
for every $X \in \bR^N$.
\end{remark}

\begin{theorem}\label{THMzero}
Recall that $M$ is the constant from \eqref{cond}. If
\begin{align}
\label{e:221115.2}
\delta > \beta \lambda_1(A) + \frac{1}{32}M^2 \lambda_1(A)^2
\end{align}
then the null solution for \eqref{sdevect},
$X(t) = X_0$, is stochastically asymptotically stable in the large in $(0,1)^N$. 
This means that  $X_0$ is stochastically stable and
\begin{align*}
 \bP \left[\lim_{t \rightarrow \infty} X(t)=0\right]=1,
\end{align*}
for all $X(0) \in (0,1)^N$.
\end{theorem}

\begin{proof}
Let us define the Lyapunov function $V: (0,1)^N \rightarrow \mathbb{R}_+=[0, \infty)$
\begin{align*}
V(X)= |X|^2;
\end{align*}
recalling the definition of the infinitesimal generator $L$ in \eqref{gen} and
setting (compare \eqref{e:090615-1})
\begin{align*}
s_i = \sum_{j=1}^N a_{ij}x_j
\end{align*}
we have
\begin{align*}
L V(X) = 2 \beta \sum_{i=1}^N x_i  s_i - 2 \delta |X|^2  -2 \beta \sum_{i=1}^N  x_i^2 s_i + \sum_{i=1}^N \sigma(x_i)^2 (1-x_i)^2 s_i^2.
\end{align*}
Since it holds that
$$x(1-x) \leq \frac14,$$
{we have} from \eqref{e:221115.1} and condition \eqref{cond} that 
\begin{equation}
L V(X) \leq \left( 2 \beta \lambda_1(A) - 2 \delta + \frac{1}{16}M^2 \lambda_1(A)^2 \right) |X|^2.
\end{equation}
In order to conclude, {we shall impose that $C= 2 \beta \lambda_1(A) - 2 \delta + \frac{1}{16}M^2 \lambda_1(A)^2$ 
is} 
strictly negative, i.e.,
\begin{align*}
\delta > \beta \lambda_1(A) + \frac{1}{32}M^2 \lambda_1(A)^2
\end{align*}
as required.
Then under this assumption we have that
\begin{align*}
L V(X) \leq C V(X).
\end{align*}
and by \cite[Theorem11.2.8)]{Arnold1974} $X_0$ is stochastically asymptotically stable in the large in $(0,1)^N$.
\end{proof}

\section{Stochastic permanence}
\label{s2.1}

We obtain, from Theorem \ref{th:1}, that the solution exists for all times and that it remains 
in $\Delta$ definitely. 
However, this property is too weak for the applications, so we search for further details about the asymptotic behavior of the solution.
First, we recall the following definition from \cite{Li2009}.

\begin{definition}
Equation \eqref{e:090615-1} (equivalently, \eqref{sdevect}) is said to be stochastically permanent if for any $\varepsilon > 0$ there exists a constant
$\chi = \chi(\varepsilon)$ such that, for any initial condition $X(0) = (x_1(0), \ldots, x_N(0)) \in \Delta$, the solution satisfies
\begin{align}
\label{e:201115.perm}
\liminf_{t \to \infty} \bP(|X(t)| \ge \chi) \ge 1-\varepsilon.
\end{align}
\end{definition}

At first, we prove a result that seems interesting on its own.

\begin{theorem}
Assume that
\begin{equation}
\label{e:201115.cond}
\delta < \lambda_1(A) \, \beta - \frac{1}{32} M^2 \lambda_1(A)^2. 
\end{equation}
Then,
for any initial condition $X(0) \in \Delta$, the solution $X(t)$ satisfies
\begin{equation}
\label{e:stima_sopra}
\sup_{t > 0} \bE\left[\frac{1}{|X(t)|^\alpha}\right] \le C
\end{equation}
where $\alpha > 0$ is small enough to have
\begin{align*}
\delta < \lambda_1(A) \, \beta - \frac{\alpha+1}{32} M^2 \lambda_1(A)^2
\end{align*}
and $C$ is a finite constant depending on 
$\alpha$, the initial condition $X(0)$, the adjacency matrix $A$ and the rates $\beta$ and $\delta$.
\end{theorem}

\begin{proof}
Let $u$ be the Perron eigenvector of the $N \times N$ adjacency matrix $A$, i.e., it is the eigenvector corresponding to the spectral radius $\lambda_1(A)$, and the unique one such that $u > 0$ and $\left\| u\right\|_1=1$ \cite{horn}.
Consider the function 
\begin{align*}
\psi(X) 
= \frac{1}{\sum\limits_{i=1}^N u_i x_i};
\end{align*}
by It\^o's formula the process $Y(t) = \psi(X(t))$ satisfies
\begin{align*}
{\rm d}Y(t) = L\psi(X(t)) \, {\rm d}t + {\rm d}M(t),
\end{align*}
where $M(t)$ is a (local) martingale and $L$ is the infinitesimal generator of the diffusion $X(t)$, defined in \eqref{gen}.
We may compute
\begin{align*}
L \psi(X) =& \sum_{i=1}^N u_i f_i(X) \partial_{x_i} \psi(X) + \frac12 \sum_{i=1}^N u_i^2 g_{ii}^2(X) \partial^2_{x_i x_i} \psi(X)
=
- \sum_{i=1}^N f_i(X) \psi^2(X) + \sum_{i=1}^N g_{ii}^2(X) \psi^3(X), 
\\
&\qquad X = (x_1, \dots, x_N), \quad \psi \in C^\infty(\Delta).
\end{align*}

Next, we introduce the process
\begin{align}
\label{e:201115.2}
Z(t) = e^{\kappa t}(1 + \psi(X(t)))^\alpha,
\end{align}
where $\kappa$ is a positive constant to be chosen later.
Again by appealing to It\^o's formula we have
\begin{equation}
\label{e:201115.0}
\begin{aligned}
{\rm d}Z(t) =& \kappa Z(t) \, {\rm d}t 
+ \alpha e^{\kappa t} (1 + \psi(X(t)))^{\alpha-1} \left[- \psi^2(X(t))
\sum_{i=1}^N u_i f_i(X(t)) +  \psi^3(X(t))\sum_{i=1}^N u_i^2 g_{ii}^2(X(t))\right] \, {\rm d}t \\
&+ \frac12 \alpha(\alpha-1) e^{\kappa t} (1 + \psi(X(t)))^{\alpha-2}
\psi^4(X(t))\sum_{i=1}^N u_i^2 g_{ii}^2(X(t)) 
\, {\rm d}t
\\
&+ {\rm d}\tilde M(t).
\end{aligned}
\end{equation}
 Let us consider

\begin{align}
\label{e:201115.1}
-\sum_{i=1}^N u_i f_i(X(t)) 
= -\sum_{i=1}^N \beta u_i s_i(t) + \sum_{i=1}^N \beta u_i s_i(t) x_i(t) + \sum_{i=1}^N \delta u_i x_i(t).
\end{align}
Since $u \ge 0$, $|u|_1=1$, from \eqref{e:221115.1} we have
\begin{align*}
\beta \sum_{i=1}^N u_i s_i(t) x_i(t) & = \beta \sum_{i,j=1}^N a_{ij} x_j(t) u_i x_i(t) \leq  \beta \sum_{i,j=1}^N a_{ij} x_j(t)  x_i(t)\\
& = \beta \langle A X(t), X(t) \rangle \le 
\beta \lambda_1(A) \, |X(t)|_2^2 \le \beta \lambda_1(A) \, \psi(X(t))^{-2},
\end{align*}
moreover 

\begin{align*}
-\sum_{i=1}^N \beta \sum_{j=1}^N u_i a_{ij} x_j(t) + \sum_{i=1}^N \delta u_i x_i(t) &=- \beta \langle u, A X(t) \rangle + \delta \langle u,  X(t) \rangle  = - \beta \langle A^Tu,  X(t) \rangle  + \delta \langle u,  X(t) \rangle \\
 & =  \left(- \beta \, \lambda_1(A) + \delta\right) \psi^{-1}(X(t)).
\end{align*}
Using these estimates in \eqref{e:201115.1} we get
\begin{align}
\label{e:201115.3}
- \sum_{i=1}^N f_i(X(t)) 
\le \beta \lambda_1(A) \, \psi^{-2}(X(t)) + \left(- \beta \, \lambda_1(A) + \delta\right) \psi^{-1}(X(t)).
\end{align}
Next, we consider 
\begin{align*}
\sum_{i=1}^N  u_i^2 g_{ii}^2 =\sum_{i=1}^N u_i^2 \left[ \sigma_i(x_i(t)) s_i(t) (1-x_i(t)) \right]^2;
\end{align*}
by Theorem \ref{th:1} we already know that $x_i(t) \in (0,1)$, then we have 
$x(1-x) \leq  1/4$,
hence the
previous sum is bounded by
\begin{align*}
\frac{M^2}{16}  \, \sum_{i=1}^N  u_i^2 \left[   \sum_{j=1}^N a_{ij} x_j(t) \right]^2 &  \le \frac{M^2}{16} \left[ \sum_{i=1}^N \left( u_i \sum_{j=1}^N  a_{ij} x_j(t)\right) \right]^2 = \frac{M^2}{16}
 \langle u, A X(t) \rangle^2 = \frac{M^2}{16} \lambda_1^2(A) \psi(X(t))^{-2},
\end{align*}


where $M$ is the constant in \eqref{cond}.
We have thus from \eqref{e:201115.0}, integrating in $(0,t)$ and taking expectation
\begin{multline}
\label{e:201115.4}
\bE [Z(t)] - \bE[Z(0)] \le 
\alpha \, \bE \int_0^t e^{\kappa s} (1 + \psi(X(s))^{\alpha-2}
\\
\cdot \left\{   \left(\beta \lambda_1(A) + \frac{\kappa}{\alpha} \right) 
+
  \left( 2 \frac{\kappa}{\alpha}  + \delta
+ \frac{M^2}{16} \lambda_1(A)^2 \right)\psi(X(s))
\right.
\\
\left.
\phantom{+ \alpha \, \bE \int_0^t e^{\kappa s} (1 + \psi(X(s))^{\alpha-2}
\{  \lambda_1(A)}
+
\left(\frac{\kappa}{\alpha} - \beta \, \lambda_1(A) + \delta + \frac{(\alpha+1)}{32}M^2 \lambda_1(A)^2\right) \psi^2(X(s)) 
 \right\} \, {\rm d}s.
\end{multline}
Choose $\kappa$ small enough to have
\begin{align*}
\delta < \lambda_1(A) \, \beta - \frac{\alpha+1}{32} M^2 \lambda_1(A)^2 - \frac{\kappa}{\alpha};
\end{align*}
notice that $\psi(X(s)) \ge \frac{1}{N}$, and the function
\begin{align*}
(1 + x)^{\alpha-2} (c_0 + c_1 x - c_2 x^2)
\end{align*}
satisfies, on that interval,
\begin{align*}
(1 + x)^{\alpha-2} (c_0 + c_1 x - c_2 x^2)
\le H < +\infty
\end{align*}
for every choice of $c_0, c_1 \in \bR$ and $c_2 > 0$ and for some positive and finite constant $H$. Thus we obtain the inequality
\begin{align*}
\bE [Z(t)] \le \bE[Z(0)] 
+ \frac{\alpha H}{\kappa} e^{\kappa t}
\end{align*}
and recalling definition \eqref{e:201115.2} it follows
\begin{align}\label{e:201115.5}
\bE[(1 + \psi(X(t)))^\alpha] \le e^{-\kappa t} \bE[Z(0)] + \frac{\alpha H}{\kappa}.
\end{align}

Next, observe the estimate $\psi^{-1}(X) = \langle u,X \rangle \le |u| |X| \le |X|$, hence
$\psi(X)^{\alpha} \ge |X|^{-\alpha}$.
Thus, by using \eqref{e:201115.5} and taking the supremum in $t > 0$,
\begin{align*}
\sup_{t > 0} \bE\left[\frac{1}{|X(t)|^\alpha}\right] \le 
\sup_{t > 0} \bE\left[ \psi^\alpha(X(t)) \right] \le
\sup_{t > 0} \bE\left[ (1 + \psi(X(t)))^\alpha \right] 
\le
 \left( \bE[(1 + \psi(X(0)))^\alpha] + \frac{\alpha H}{\kappa} \right)
\end{align*}
as required.
\end{proof}

The main result in this section is the following.

\begin{theorem}\label{THMperm}
Assume that condition \eqref{e:201115.cond} holds.
Then the solution of the system \eqref{sdevect} is stochastically permanent.
\end{theorem}

\begin{proof}
The proof follows from a simple application of Markov's inequality.
Let us estimate
\begin{align*}
\bP(|X(t)| < \chi) 
\end{align*}
for some $\chi$ to be chosen. Then
\begin{align*}
\bP(|X(t)| < \chi) 
= \bP \left(\frac{1}{|X(t)|} > \frac{1}{\chi} \right) \le \frac{\bE\left[1/|X(t)|^\alpha\right]}{1/\chi^\alpha}
\le C \chi^\alpha,
\end{align*}
where $C$ is the constant from \eqref{e:stima_sopra}. 
The above inequality holds by taking the supremum:
\begin{align*}
\sup_{t > 0} \bP(|X(t)| < \chi) 
\le C \chi^\alpha,
\end{align*}
and therefore
\begin{align*}
\inf_{t > 0} \bP(|X(t)| \ge \chi) \ge 1- C \chi^\alpha.
\end{align*}
Since for every $\varepsilon > 0$ we can
find $\chi = (\varepsilon/C)^{1/\alpha}$, inequality \eqref{e:201115.perm} is satisfied, as required.
\end{proof}

{By appealing to Theorems \ref{THMzero} and \ref{THMperm},
we can formulate both the conditions, proving that the null solution is asymptotically stable, and that the system is stochastically permanent in terms of the ratio
$\beta/\delta$.} Thus, the null solution is asymptotically stable in the large provided that
\begin{equation}\label{thrstoc}
\tau=\frac{\beta}{\delta} < \tau_c^s := \frac{1}{\lambda_1(A)} - \frac{M^2 \lambda_1(A)}{32 \delta} = \tau_{c}^{(1)} - \frac{M^2 \lambda_1(A)}{32 \delta}.
\end{equation}
The solution of \eqref{sdevect} is stochastically permanent if
\begin{align}
\tau= \frac \beta\delta > \tau^s_p := \frac{1}{\lambda_1(A)} + \frac{M^2 \lambda_1(A)}{32 \delta}=\tau_{c}^{(1)}+ \frac{M^2 \lambda_1(A)}{32 \delta},
\end{align}
{where $\tau_{c}^{(1)}$ is the epidemic threshold of the NIMFA model (see Section \ref{mfa}).} 
{We see that 
there is a gap between the regions where the effective infection rate $\tau$ leads to extinction or persistence, respectively, whose extension depends on the intensity of the noise, through the parameter $M$.}
We underline however that both Theorem \ref{THMzero} and \ref{THMperm} give us only sufficient conditions.

\begin{figure}[t]	
\centering
\includegraphics[width=0.8\textwidth]{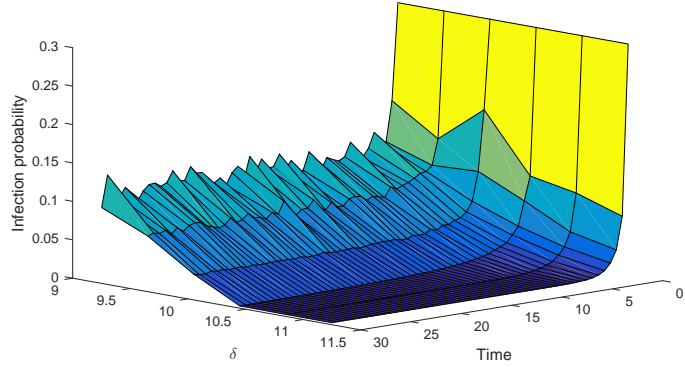}
\caption{{Dynamics of the infection probability of the node $4$ for different values of $\delta$, with fixed vales of $\beta=0.2$ and $M=0.08$. The sample network has an arbitrary
topology with $N=80$}.}\label{fig:surf}
\end{figure}

\begin{figure}[h]
        \centering
        \begin{minipage}[c]{.3\textwidth}
          \subfigure[{}]
         { \includegraphics[width=\textwidth]{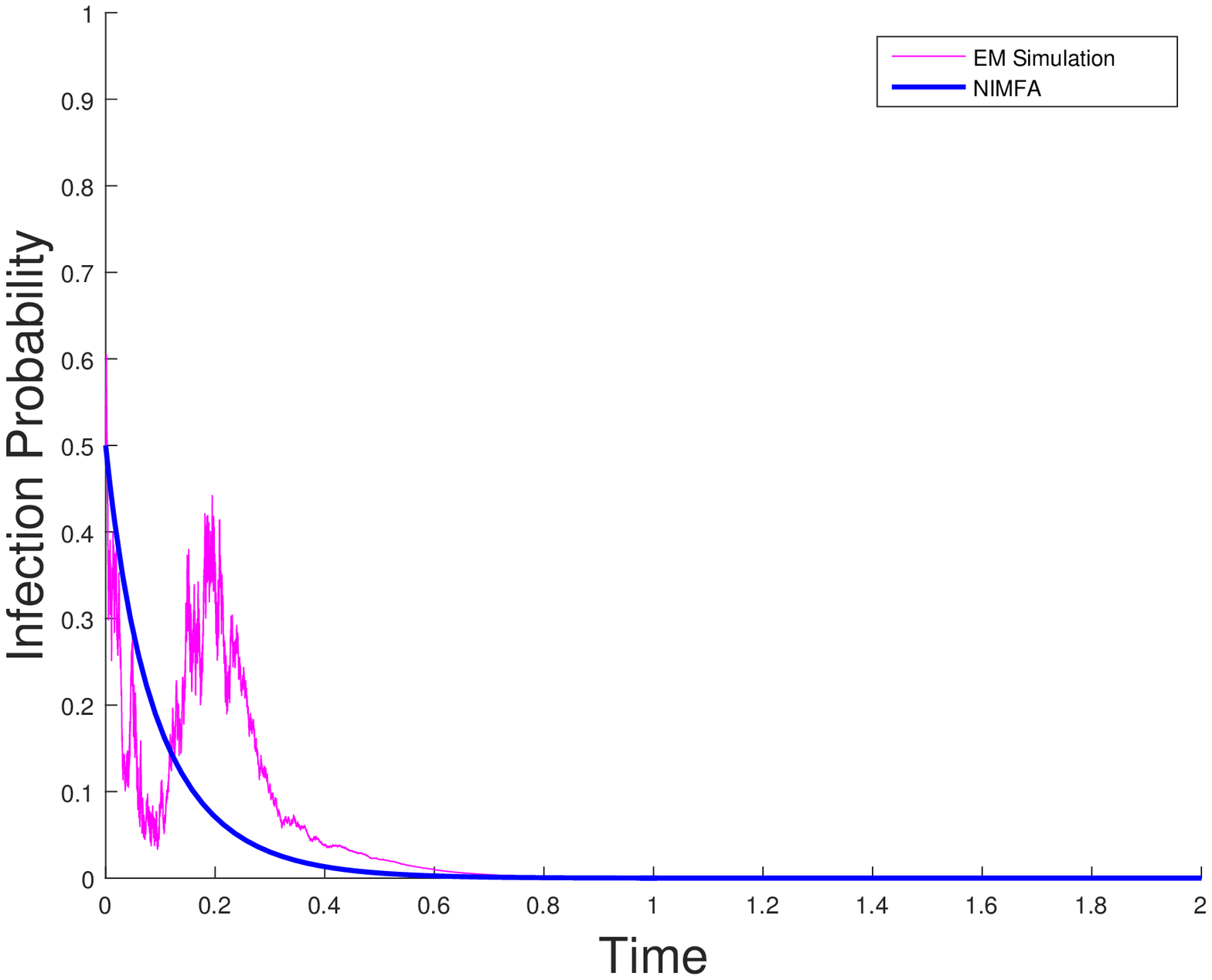}}
        \end{minipage}%
\hspace{-0.5cm}
      \begin{minipage}[c]{.3\textwidth}
	 
         \subfigure[{}]
         {  \includegraphics[width=\textwidth]{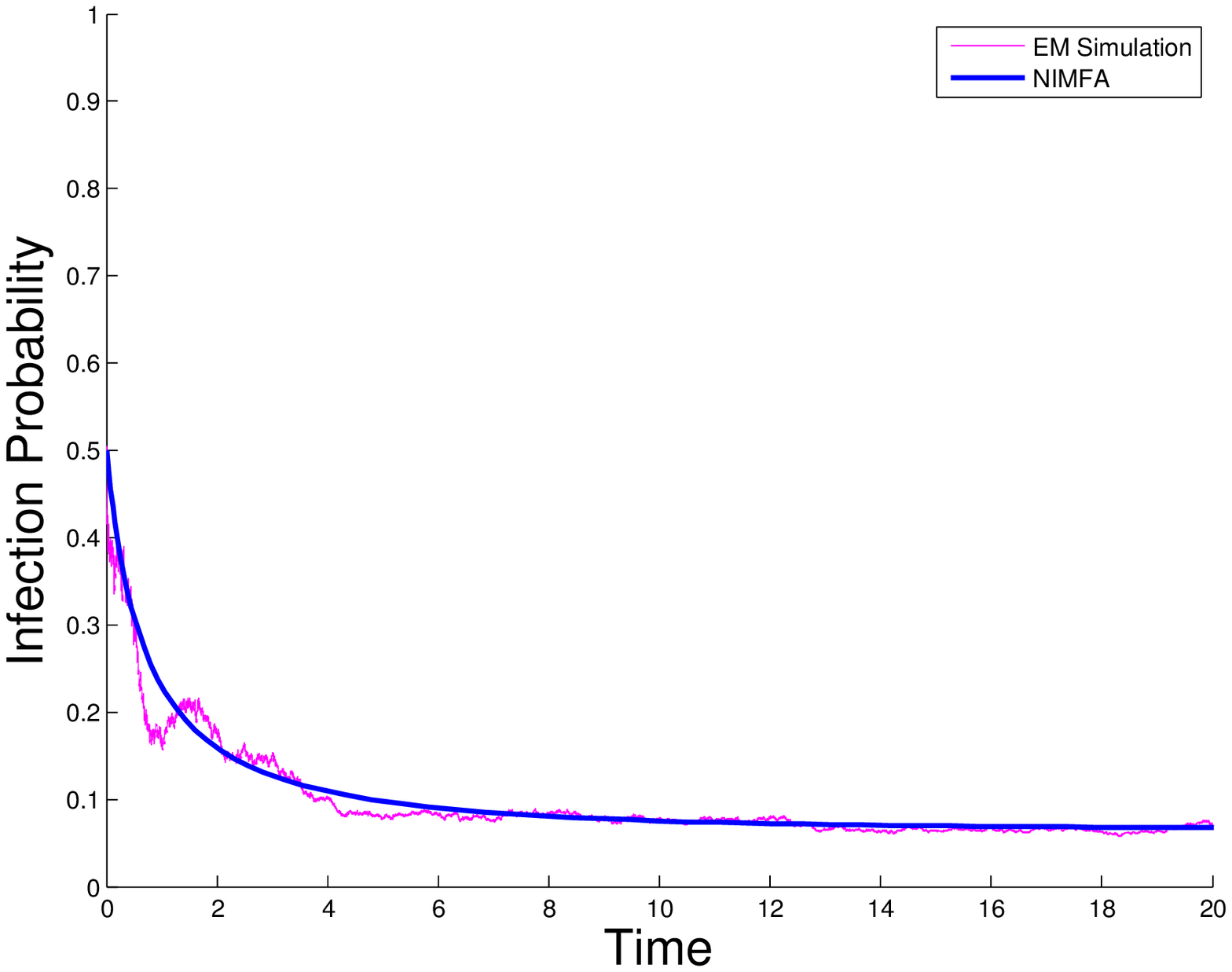}}
      \end{minipage}
	 \hspace{-0.5cm}
	 \begin{minipage}{.3\textwidth}
	 \subfigure[{}]%
	 {\includegraphics[width=\textwidth]{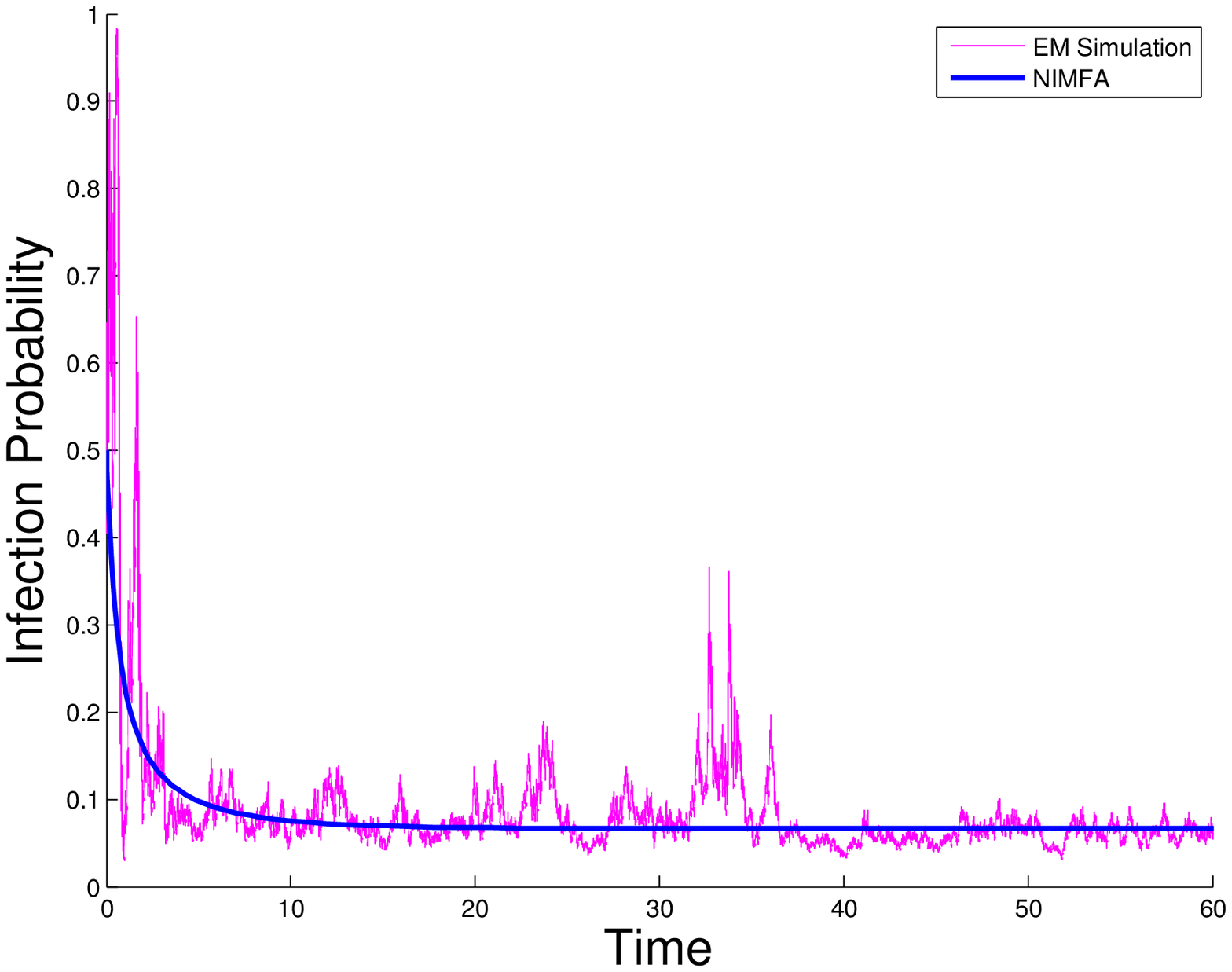}}
\end{minipage}

        \caption{ Dynamics of the infection probability of the node $4$ in a graph
with ring topology, and $N=50$, where $\tau_c^{(1)}=0.5$: EM approximation of the solution of \eqref{sdevect} versus
solution of \eqref{nimfa}. 
a) $\beta=4.1$, $\delta=16.3$, $M=8$, $\beta/\delta < \tau_c^s=0.7454$.  b) $\beta=1.5$, $\delta=2.8$, $\beta/\delta > \tau_p^s=0.5143$, M=0.8. c)$\beta=1.5$, $\delta=2.8$, $\beta/\delta >\tau_p^s=0.8571$, M=4. }\label{fig:2}
      \end{figure}

\begin{figure}[h]
        \centering
        \begin{minipage}[c]{.3\textwidth}
          \subfigure[{}]
         { \includegraphics[width=\textwidth]{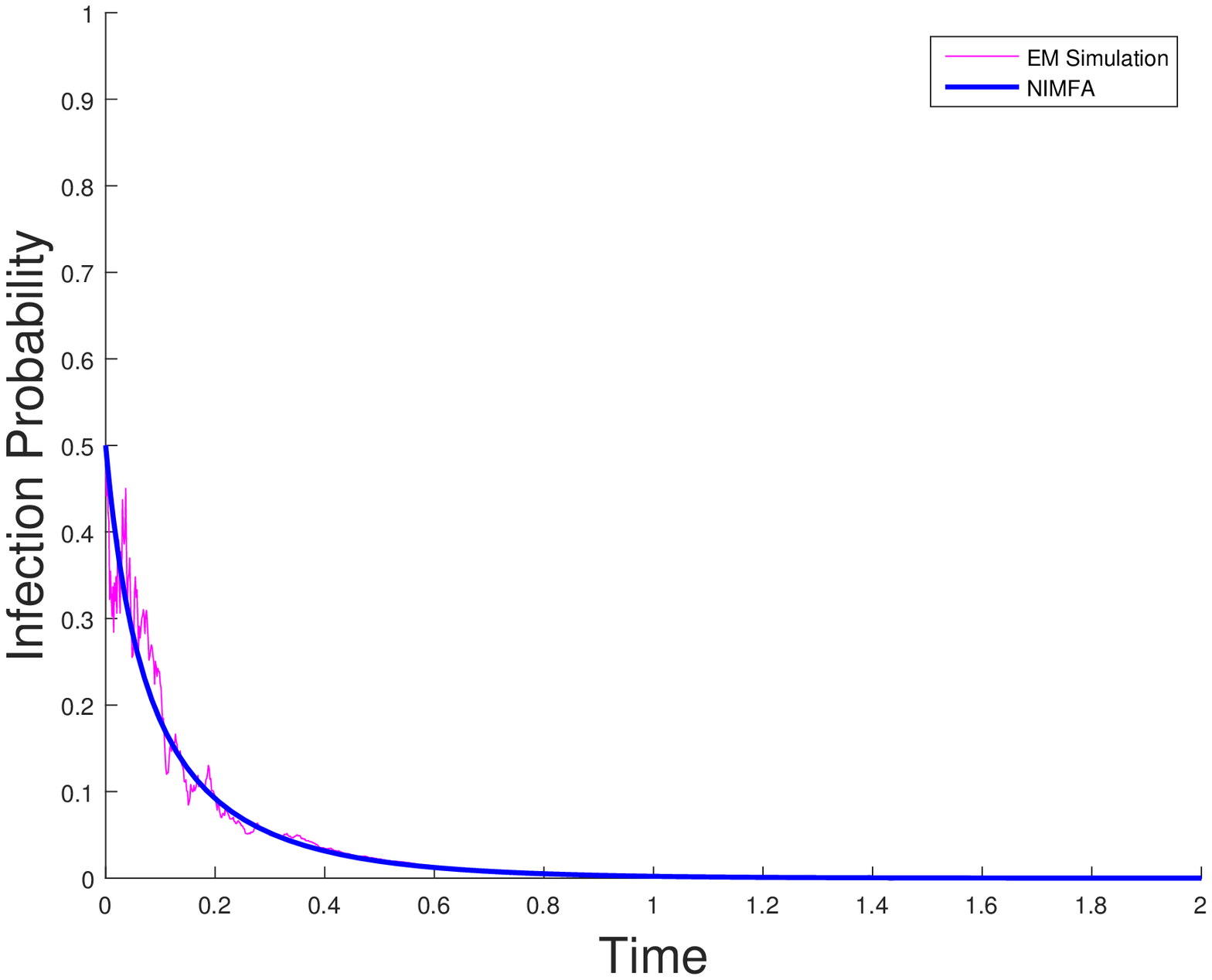}}
        \end{minipage}%
\hspace{-0.5cm}
        \begin{minipage}[c]{.34\textwidth}
	 
         \subfigure[{}]
        {\includegraphics[width=\textwidth]{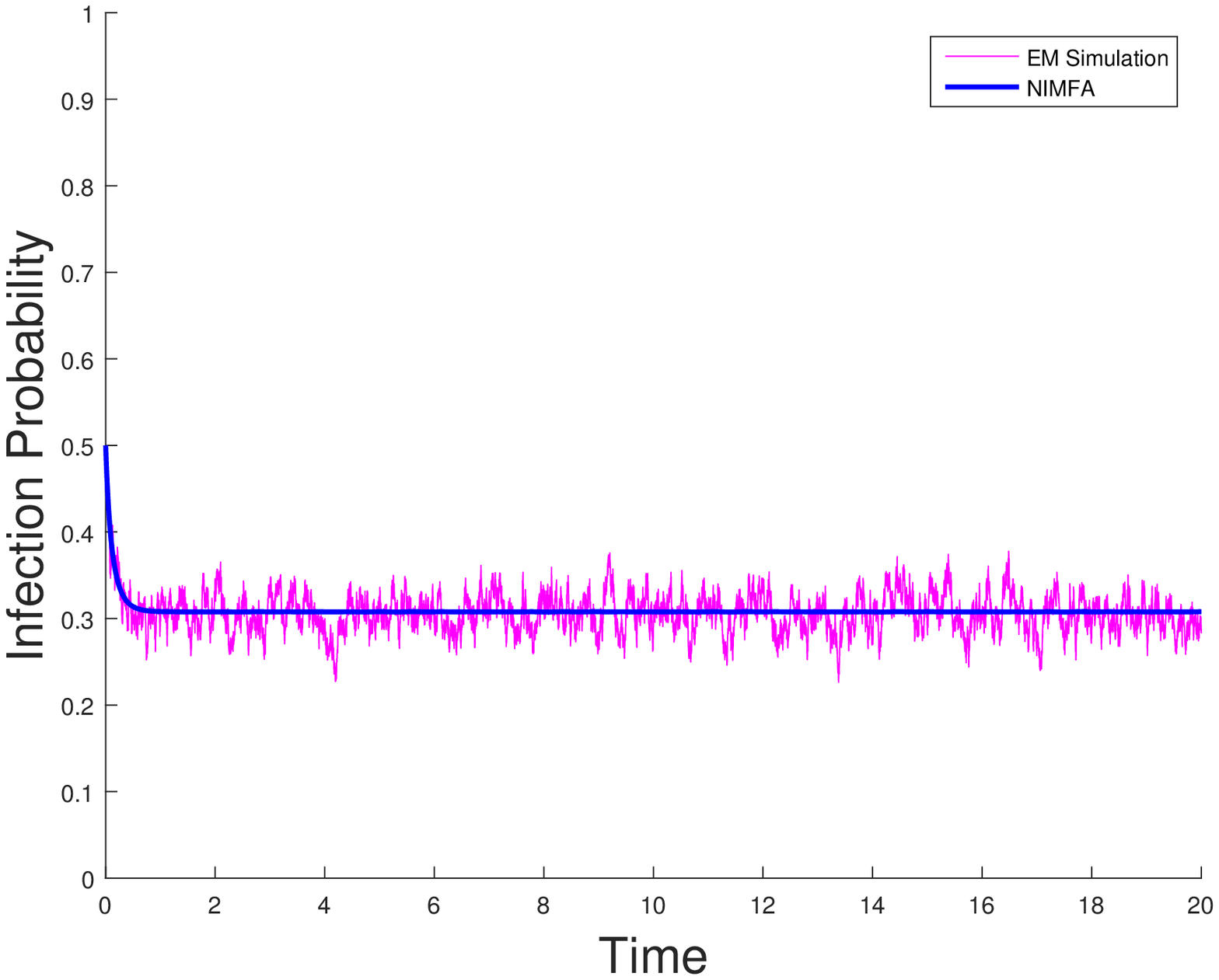}}
        \end{minipage}
	 \hspace{-0.5cm}
	 \begin{minipage}{.34\textwidth}
	 \subfigure[{}]%
	 {\includegraphics[width=\textwidth]{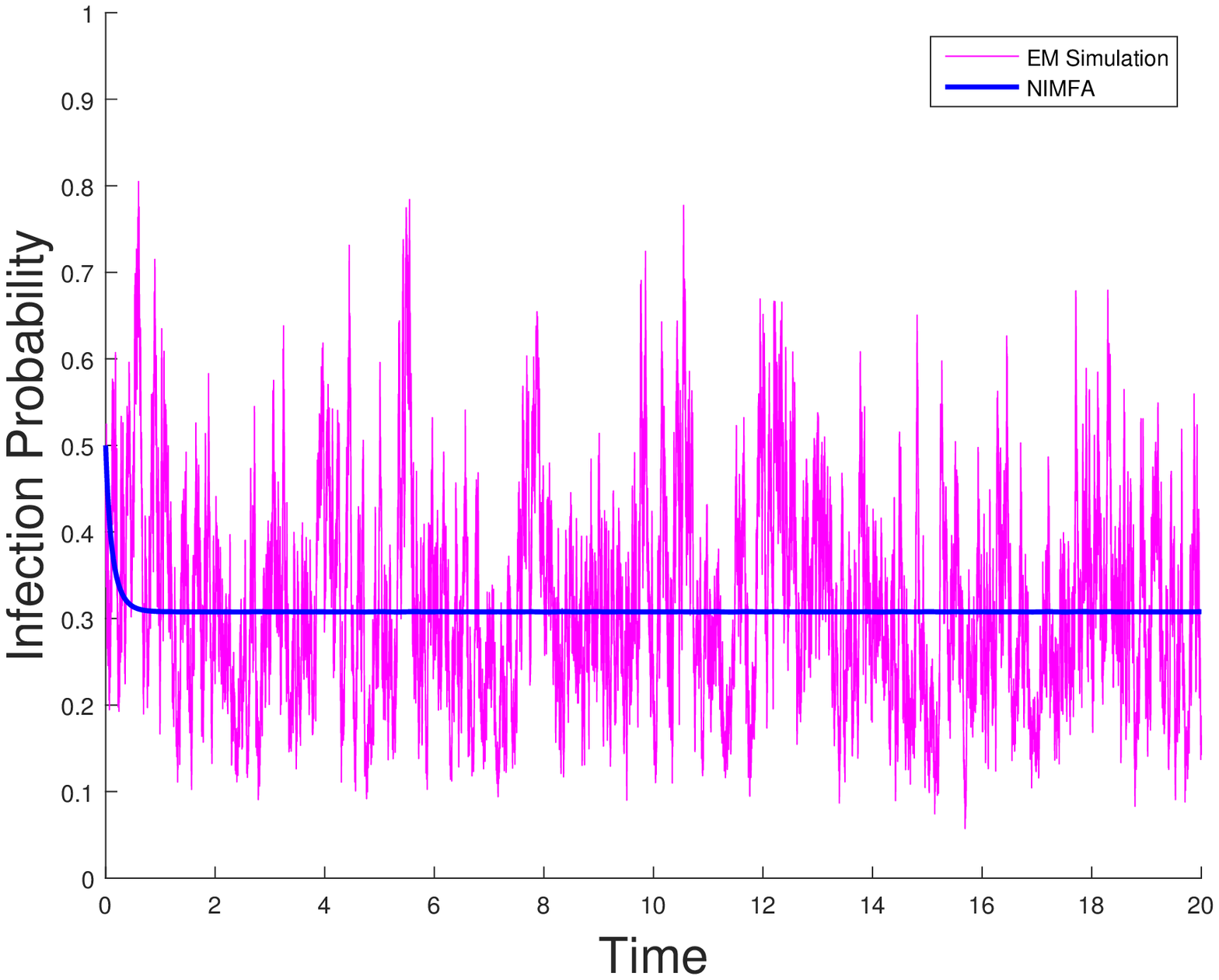}}
\end{minipage}

        \caption{ Dynamics of the infection probability of the node $4$ in a complete graph
with $N=40$, where $\tau_c^{(1)}=0.0256$: EM approximation of the solution of \eqref{sdevect} versus
solution of \eqref{nimfa}. 
 a) $\beta=0.5$, $\delta=23.9$, $\beta/\delta \leq \tau_c^s=0.0210$,  $M=0.3$.  b) $\beta=0.5$, $\delta=13.5$, $\beta/\delta > \tau_p^s=0.0258$, $M=0.04$. c) $\beta=0.5$, $\delta=13.5$, $\beta/\delta > \tau_p^s=0.0338$, M=0.3}\label{fig:3}
      \end{figure}

\begin{figure}[!h]
        \centering
        \begin{minipage}[c]{.3\textwidth}
          \subfigure[{}]
         { \includegraphics[width=\textwidth]{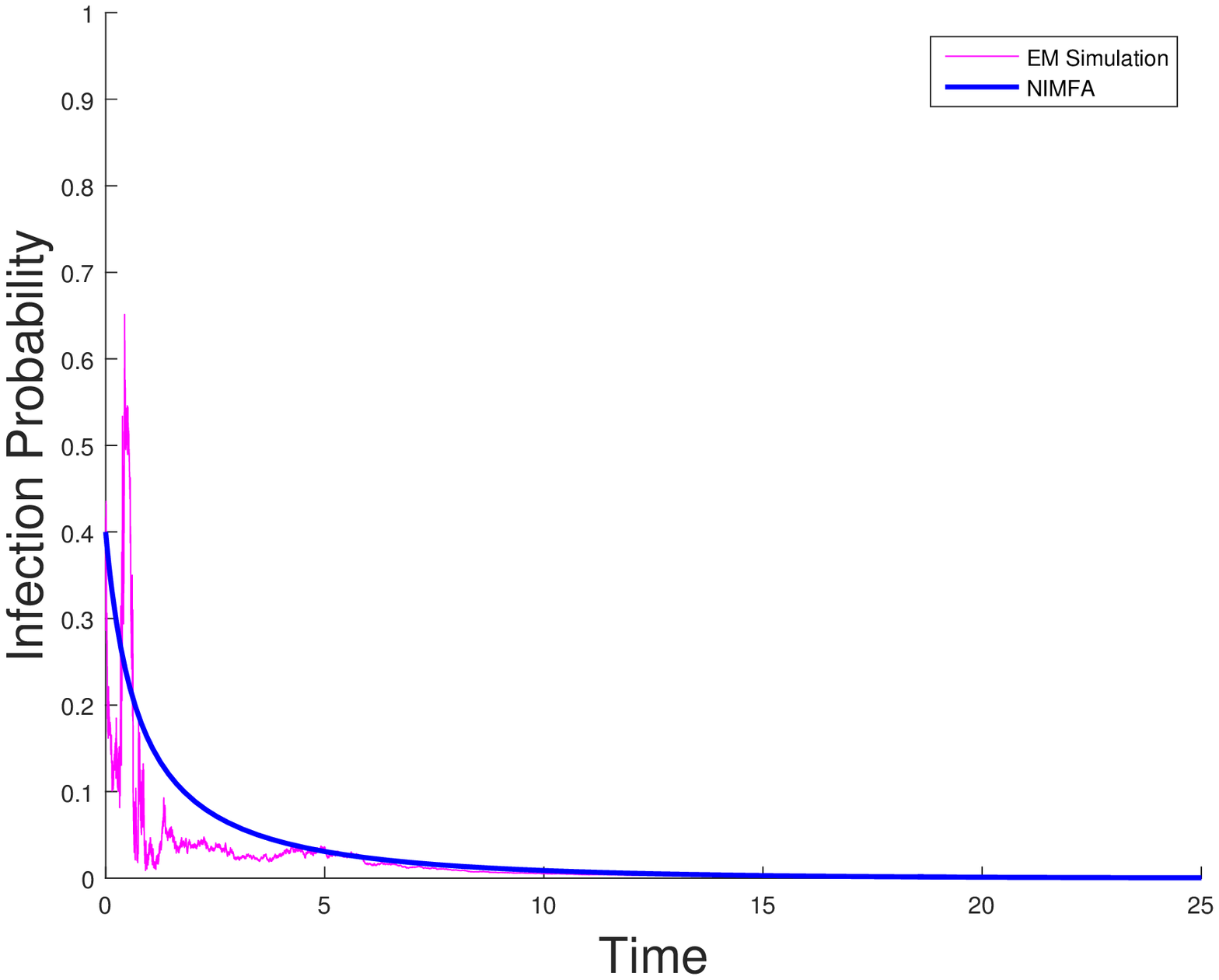}}
        \end{minipage}%
\hspace{-0.5cm}
        \begin{minipage}[c]{.3\textwidth}
	 
         \subfigure[{}]
         {  \includegraphics[width=\textwidth]{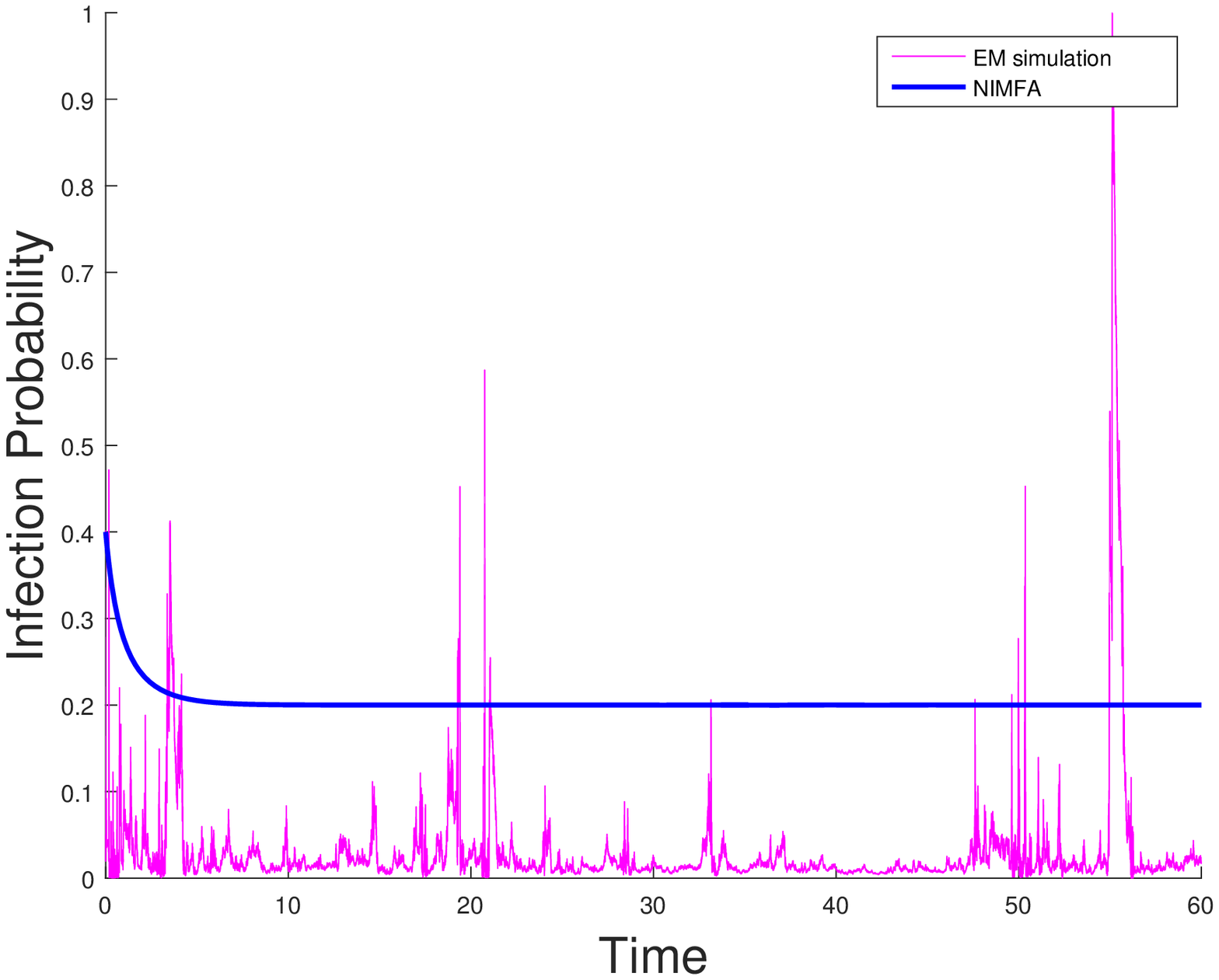}}
        \end{minipage}
	 \hspace{-0.5cm}
	 \begin{minipage}{.3\textwidth}
	 \subfigure[{}]%
	 {\includegraphics[width=\textwidth]{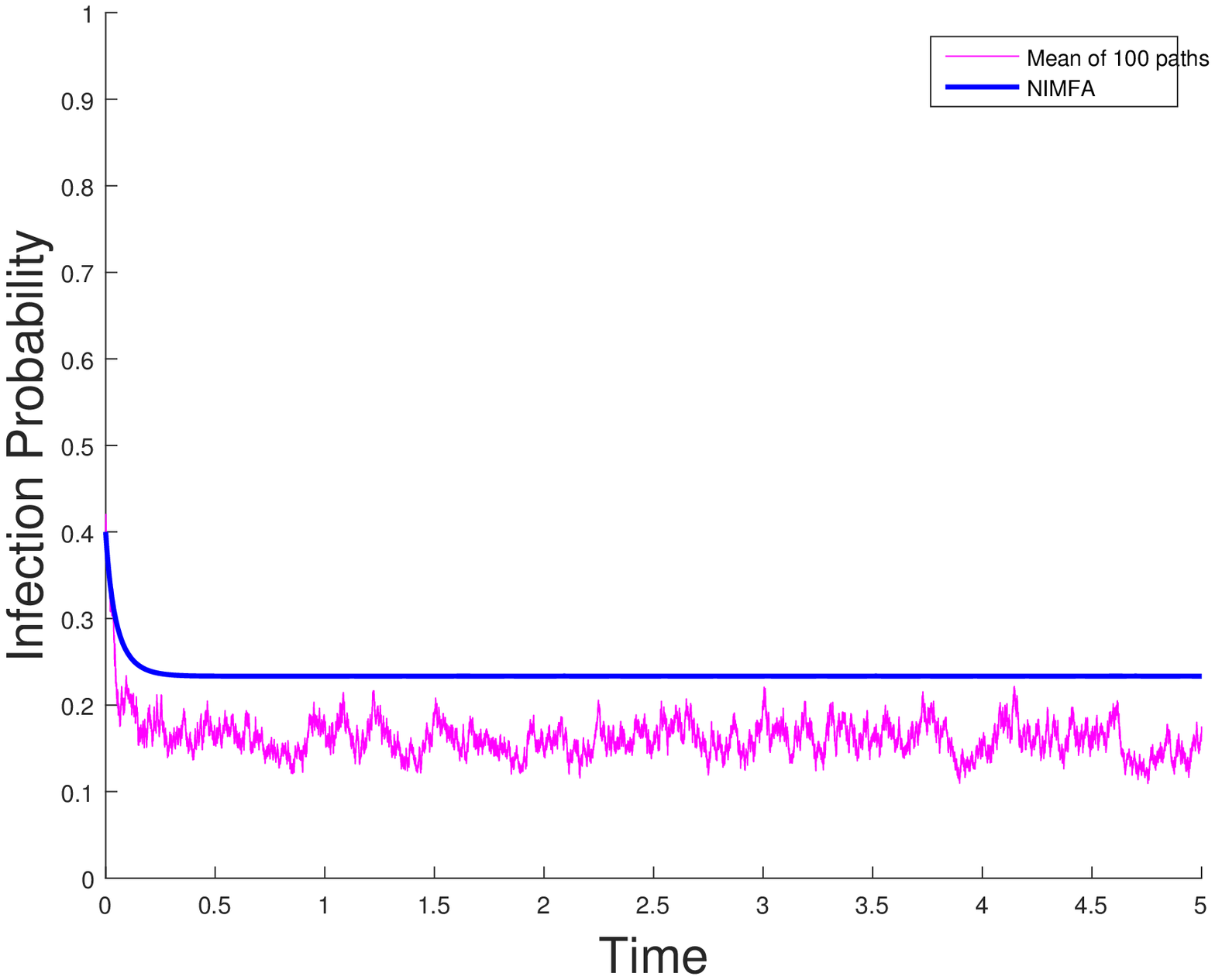}}
\end{minipage}

        \caption{  Dynamics of the infection probability of the node $4$ for a graph with ring topology, 
and $N = 50$, where $\tau_c^{(1)}=0.5$: EM approximation of the solution of \eqref{sdevect} versus
solution of \eqref{nimfa}. 
a) $\beta=1.5$, $\delta=3.2$, $M=10$,  $0<\beta/\delta < \tau_1^{(c)}$. b) $\beta=1.5$, $\delta=2.4$, $M=40$,  $ \tau_1^{(c)} <\beta/\delta < \tau_p^s$. c) EM approximation of the solution of \eqref{sdevect} averaged over 100 sample paths versus. $\beta=30$, $\delta=46$, $M=30$, $\tau_c^{(1)} <\beta/\delta < \tau_p^s$. }\label{fig:4}
      \end{figure}

\section{Numerical experiments}\label{num}

{We numerically simulate the solution of the system \eqref{sdevect} by the Euler-Maruyama (EM) method \cite{higham2001}.
{In Figure \ref{fig:surf} we consider a graph with an arbitrary topology with $N=80$ and average degree equals to $49$. We plot the dynamics of the infection probability of a given node, along one sample path, for different values of $\delta$, and fixed values of $\beta=0.2$ and $M=0.08$. 
Here it is clear how the behavior in time of the solution changes with respect the value of $\delta$, once fixed the mean value of the infection rate and the level of noise.
Indeed, the values  $\delta=10.8$ and $\delta=11.2$ satisfy the condition \eqref{e:221115.2} and, as we expect, the solution approaches the disease-free equilibrium. The values of $\delta=9.2$ and $\delta=9.6$, instead, satisfy the condition \eqref{e:201115.cond} and we can see that the solution appears stochastically permanent. The intermediate values of $\delta=10$ and $\delta=10.4$ do not satisfy any of the two sufficient conditions for extinction and persistence respectively. We see that for $\delta=10$ the solution stays positive, instead for $\delta=10.4$ the solution approaches the zero point.}

{In Figure \ref{fig:2} 
(a), (b) and (c) 
we compare the solution of the system \eqref{sdevect} with the solution of the NIMFA system \eqref{nimfa}.
 We consider a graph with ring topology (each node has two neighbors) and $N=50$.}
In (a) we consider values of $\beta$ and $\delta$ such that $\tau < \tau_c^s$ with $M=8$, 
and we plot the dynamical behavior of one single selected node, by computing the solution of \eqref{sdevect} along one sample path.
The numerical computation confirms the stability result in Theorem \ref{THMzero}. 
In (b) and (c), instead, we consider values of $\beta$ and $\delta$ such that $\tau >\tau_p^s$ for $M=0.8$ and $M=4$ respectively. 
We can recognize the behavior aforesaid in Theorem \ref{THMperm}. 
Moreover we can see that, if the assumption \eqref{e:201115.cond} of Theorem \ref{THMperm} holds, the solution of \eqref{sdevect} fluctuates around the endemic equilibrium of the
system \eqref{nimfa} and, clearly, with the decrease of the intensity of the noise, the fluctuations are smaller.
The same type of numerical experiments have been done in Figure \ref{fig:3} 
(a), (b) and (c), for a complete graph (all nodes are connected among themselves) and $N=50$.

In Figures \ref{fig:4} 
and \ref{fig:5} 
we investigate the behavior of the solution of \eqref{sdevect}, in the case where both
 conditions of stability \eqref{e:221115.2} and permanence \eqref{e:201115.cond} are not satisfied. 
 Precisely, in Figure \ref{fig:4} 
 we consider the graph with ring topology and $N=50$; 
 in particular, in (a) we consider the case where $0<\beta/\delta < \tau_c^{(1)}$
 and we can see that the solution of \eqref{sdevect} tends to zero, as that of the deterministic system \eqref{nimfa}. 
 In Figure \ref{fig:4} 
 (b), instead, we analyze the case $ \tau_c^{(1)} <\beta/\delta < \tau_p^s$, we can observe 
that the solution of \eqref{sdevect} does not fluctuate around the solution of \eqref{nimfa}. Then, in (c) the EM solution is averaged over 100 sample paths always in the case $ \tau_c^{(1)} <\beta/\delta < \tau_p^s$; we can see that, in this case, NIMFA provides an upper bound of our infection dynamics.   
  The same behavior, in the region $ \tau_c^{(1)} <\beta/\delta < \tau_p^s$, of one sample path, and of the averaged solution,  is depicted also by Figure \ref{fig:5} 
  (a) and (b) respectively, where we consider a graph with an arbitrary topology and $N=13$.

%

%

\begin{figure}[!ht]
        \centering
        \begin{minipage}[c]{.40\textwidth}
          \subfigure[{}]
         { \includegraphics[width=\textwidth]{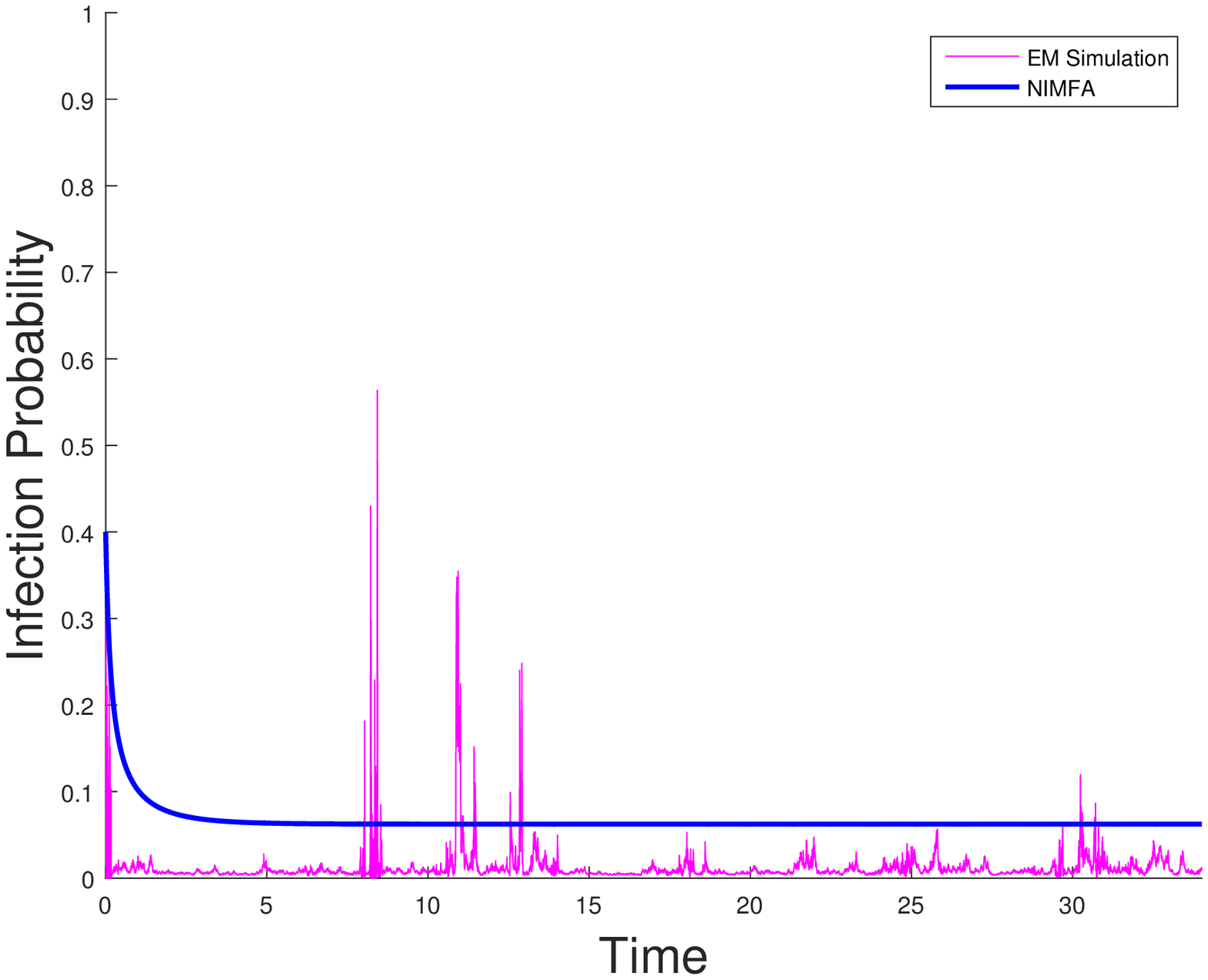}}
        \end{minipage}%
\hspace{.5cm}
        \begin{minipage}[c]{.40\textwidth}
	 
         \subfigure[{}]
         {  \includegraphics[width=\textwidth]{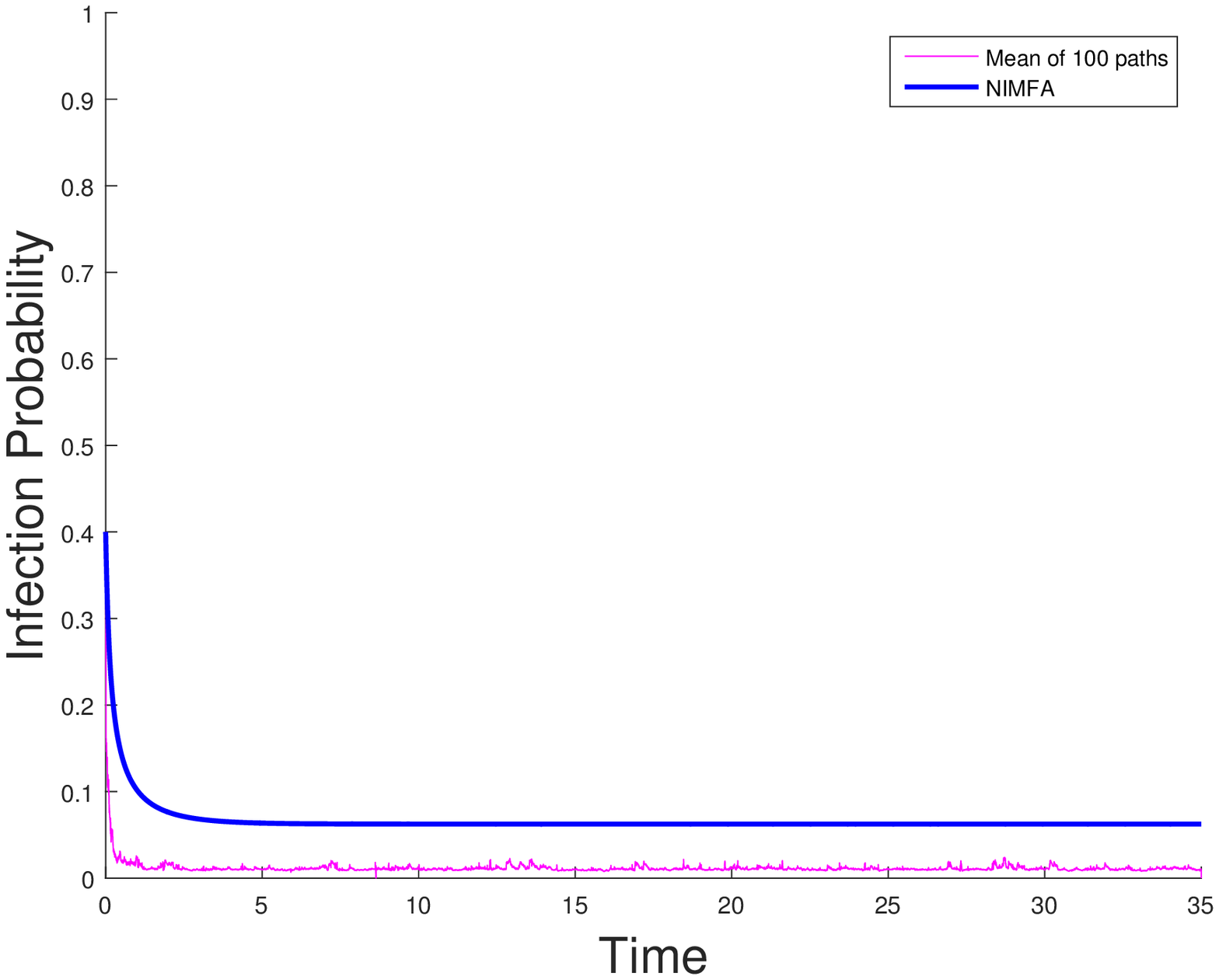}}
        \end{minipage}
	         \caption{   Dynamics of the infection probability of the node 4 in a graph with arbitrary topology and $N=13$, where $\tau_c^{(1)}=0.2045$. 
					$\beta=2.2$, $\delta=10$, $M=40$, $\tau_c^{(1)}<\beta/\delta < \tau_p^s$. 
(a)  EM approximation of \eqref{sdevect} versus
solution of \eqref{nimfa}.  (b)  EM approximation of the solution of \eqref{sdevect} averaged over 100 sample paths versus
solution of \eqref{nimfa}.}
\label{fig:5}
      \end{figure}

\section{Conclusion}

{The aim of the this paper is to investigate the behavior of  epidemics spreading in a population with inhomogeneous contact rates, where the rates at which each individual can be infected from its neighbors are considered as independent stochastic processes.} 

 {Our idea is to start from the deterministic system \eqref{nimfa}, obtained after a mean-field approximation first proposed in \cite{VanMieghem2009}, where the infection rate $\beta$ between each two given individuals is either zero, if they are not in contact, or
 a given constant, if they are connected. 
 However, since epidemic processes are usually affected by random disturbances, we  introduce in this model a stochastic heterogeneity of the population
 by taking into account a variability in time of the parameters. 
 Precisely, with respect to the last point, we 
 assume that the rate of receiving the infection, for each individual, varies around a common average value under the action of a family of independent,
 identically distributed Brownian motions.
This heterogenity may depend, e.g., on the state of the immune system or, in the case of diffusion of opinions, on the characterial propensity to get involved with other people's ideas.}  

{As opposite to our previous models (see \cite{Bonaccorsi2014a, Bonaccorsi2015}) we do not consider intermediate structures, like households, hospitals, cities, airports, etc.,
but we hope to return to this problem in a subsequent paper. 
Another improvement of the work would be to consider a model that allows for dynamic communities demographics, to take different average values of the infection rate for each node, 
 and to consider other types of environmental noise, like the telegraph noise.}

{With respect to the stochastic system \eqref{e:090615-1}, we  prove that it possesses a unique global
solution that remains within $(0,1)^N$ whenever it starts from this region.
Then we concentrate on the asymptotic behavior of the solution.
We may show that, if we are
in a certain region parameters,
 the solution tends to extinction almost surely. 
On the opposite,  we have discussed on stochastic permanence of the solution, finding a condition under which the epidemic process is stochastically permanent.}
{The two regions of extinction and permanence are, unfortunately, not adjacent, as there is a gap between them, whose extension
depends on the specific level of noise. 
In this intermediate region, we have performed numerical simulations to test the long time behavior of the system.}\\
{The numerical experiments confirms the analytical results in both regions of permanence and extinction. In particular, in the region of permanence, it has been seen that the solution of the stochastic system fluctuates around the positive equilibrium point of the deterministic system \eqref{nimfa}, clearly the fluctuations increases with the value of the noise level.}\\
{In the intermediate region, we have seen that, our system tends to have the same asymptotic behavior of the
deterministic model \eqref{nimfa}, however for possible large deviations from the mean infection rate, above the NIMFA threshold, the solution of our stochastic system does not seem to fluctuate around the positive equilibrium point of \eqref{nimfa}, and we can observe some differences in the average level of infection.}

 {To the best of our knowledge this is one of the first attempts to consider the parameters of the epidemic model as stochastic processes, in the case of heterogeneous networked population, thus it is clear that many improvements of the model could be possible.  Since the topic is of high practical relevance, clearly it is fundamental to conduct further investigations on the
interplay between parameter drift, noise, population contact network and heterogeneity for
 epidemic models 
\cite{widder2014heterogeneous}}.
%

\section*{Bibliography}

\def\cprime{$'$}

\end{document}